\documentclass{amsart}
\date{\today}
\usepackage{graphicx} 
\usepackage[dvipsnames]{xcolor}
\usepackage{verbatim,enumerate}
\title[Mixed type surfaces 
with  bounded mean curvature]{%
Mixed type surfaces 
with  bounded mean curvature
in $3$-dimensional space-times 
}
\author{A.~Honda}
\author{M.~Koiso}
\author{M.~Kokubu}
\author{M.~Umehara}
\author{K.~Yamada}
\address[Atsufumi Honda]{
  National Institute of Technology, Miyakonojo College, 
  473-1, Yoshio-cho, Miyakonojo, 
  Miyazaki 885-8567,
  Japan}
\email{atsufumi@cc.miyakonojo-nct.ac.jp}

\address[Miyuki Koiso]{%
Institute of Mathematics for Industry, 
Kyushu University, 
744 Motooka, Nishi-ku, 
Fukuoka, 819-0395, 
Japan 
}
\email{
koiso@math.kyushu-u.ac.jp
}

\address[Masatoshi Kokubu]{%
   Department of Mathematics, School of Engineering, 
   Tokyo Denki University, 
   Tokyo 120-8551, Japan}
\email{kokubu@cck.dendai.ac.jp}

\address[Masaaki Umehara]{
  Department of Mathematical and Computing Sciences,
  Tokyo Institute of Technology,
  2-12-1-W8-34, O-okayama Meguro-ku,
  Tokyo 152-8552, Japan}
\email{umehara@is.titech.ac.jp}

\address[Kotaro Yamada]{%
  Department of Mathematics,
  Tokyo Institute of Technology,
  O-okayama, Meguro, Tokyo 152-8551, Japan%
}
\email{kotaro@math.titech.ac.jp}


\thanks{}
  
\numberwithin{equation}{section}
\newtheorem{Thm}{Theorem}[section]

\newtheorem{Cor}[Thm]{Corollary}
\newtheorem{Lemma}[Thm]{Lemma}
\newtheorem{Prop}[Thm]{Proposition}
\theoremstyle{definition}
\newtheorem{Def}[Thm]{Definition}
\newtheorem{Exa}[Thm]{Example}
\theoremstyle{remark}

 \newtheorem*{ack}{Acknowledgements}
\newcommand{\vect}[1]{\boldsymbol{#1}}
\newcommand{\op}[1]{{\operatorname{#1}}}

\newcommand{\R}{{\vect{R}}}

\newcommand{\mb}[1]{\vect{#1}}
\newcommand{\pmt}[1]{{\begin{pmatrix} #1  \end{pmatrix}}}

\renewcommand{\phi}{\varphi}
\renewcommand{\epsilon}{\varepsilon}

\begin{document}
\maketitle
\begin{abstract}
In this paper, we shall prove that
space-like surfaces with bounded mean
curvature functions
in real analytic Lorentzian
3-manifolds can
change their causality to
time-like surfaces
only if the mean curvature functions
tend to zero.
Moreover, we shall show
the existence of such surfaces 
with non-vanishing mean curvature and
investigate their properties.
\end{abstract}

%%%%%%%%%%%%%%%%%%%%%%%%%%%%%%%%%%%%%%%%%%%%%%%%
\section{Introduction}
\label{intro}
We say that a connected surface immersed in 
a Lorentzian 3-manifold $(M^3,g)$
is of {\it mixed type} 
if both the space-like and time-like point sets
are non-empty.
In general, the mean curvature of
such surfaces 
diverges: for example,
the graph of a smooth function
$t=f(x,y)$ in 
the Lorentz-Minkowski space-time
$(\R^3_1;t,x,y)$ 
gives a space-like (resp. time-like)
surface if $B>0$ (resp. $B<0$),
where
\begin{equation}\label{eq:B}
B:=1-f_x^2-f_y^2.
\end{equation}
In this situation,
the unit normal vector is given by
\begin{equation}\label{eq:nu}
\nu=\frac{1}{\sqrt{|B|}}
(1, f_x,f_y),
\end{equation}
and the mean curvature function is
computed as
\begin{equation}\label{eq:H0}
H=
\frac{\left(f_x^2-1\right)f_{yy} -2 f_x f_y f_{xy}
+\left(f_y^2-1\right) f_{xx}}
{2 |B(x,y)|^{3/2}},
\end{equation}
which is unbounded around the
set $\{B(x,y)=0\}$, in general.

On the other hand, several zero 
mean curvature surfaces 
of mixed type in $\R^3_1$ 
were found in  
\cite{K1}, \cite{G}, \cite{Kl}, \cite{ST}, \cite{KKSY}, 
\cite{FRUYY2}, \cite{FKKRSUYY1} and \cite{FKKRSUYY2}.
Moreover, such examples can be found in other 
space-times:
in fact, a zero mean curvature surface of
mixed type in the
de Sitter 3-space (resp. in the anti-de Sitter 
3-space) is given
in this paper (cf. Example \ref{ex:fds} and 
Example \ref{ex:fads}).
It is known that zero mean 
curvature surfaces in $\R^3_1$
change types across their fold singularities,
except for the special case as in \cite{FKKRSUYY1}.
On the other hand, in \cite{HKS}, it was shown that
space-like non-zero constant mean curvature surfaces 
do not admit fold singularities,
which suggests that 
space-like non-zero constant mean curvature surfaces 
never change types.  More precisely,
the following questions naturally arise:
\begin{itemize} 
\item[(a)] {\it Is there a mixed type
surface with non-zero constant  mean curvature? }
\item[(b)] {\it Is there a mixed type
surface whose mean curvature vector field
is smooth and does not vanish
along the curve of type change?} 
\end{itemize} 
In this paper, we show that 
the answer to Question (a) 
is negative. This is a consequence of
the following assertion:

\begin{Thm}\label{thm:main}
Let $U$ be a connected domain in $\R^2$, and
$f:U\to (M^3,g)$ a real analytic 
immersion into an oriented real analytic 
Lorentzian manifold
$(M^3,g)$.
We denote by $U_+$ $($resp. $U_-)$
the set of points where
$f$ is space-like $($resp. time-like$)$.
Suppose that $U_+,U_-$ are both non-empty,
and the mean curvature function $H$
on $U_+\cup\, U_-$ is bounded.
Then for each $p\in \overline{U_+}\cap \overline{U_-}$,
there exists a sequence $\{p_n\}_{n=1,2,3,..}$
in $U_+$ $($resp. $U_-)$
converging to $p$
so that $\lim_{n\to \infty}H(p_n)=0$, where $\overline{U_+},\overline{U_-}$
are the closures of $U_+,U_-$ in $U$.
\end{Thm}

There exist space-like and time-like
constant mean curvature 
immersions in $\R^3_1$ which are 
not of mixed type although
their induced metrics degenerate
along certain smooth curves
(cf. Examples \ref{ex:fp} and \ref{ex:fh}
in Section 2).
Also, there are similar such examples of
space-like  constant mean curvature one
surfaces in 
the
de Sitter 3-space $S^3_1$ 
with singularities which are not of mixed type 
(\cite{FKKRUY}).
The existence of such examples implies that
we cannot drop the assumption that 
both
$U_+,U_-$ are non-empty.
The proof of Theorem \ref{thm:main} is given 
in Section 2.

On the other hand, we show that 
the answer to Question
(b) 
is affirmative.
In fact, we show in Section 3 that 
the mean curvature vector fields 
of real analytic surfaces of mixed type
with bounded mean curvature functions
can be analytically
extended across the sets of type change
under a suitable genericity assumption
(cf. Proposition \ref{prop:tc}).
Moreover, we show the following:

\begin{Thm}\label{thm:existence}
There exists a real analytic function $g(x,y)$
on $\R^2$ whose graph realized in $\R^3_1$ 
satisfies the following properties:
\begin{enumerate}
\item the set $\Sigma_g$ of non-degenerate points of
type change of the graph of $g$ is non-empty, 
and the induced metric of the graph of $g$ 
is non-degenerate on $\R^2\setminus \Sigma_g$,
\item the mean curvature function of
the graph of $g$ is bounded on $\R^2\setminus \Sigma_g$.
\item 
the mean curvature vector field can be
extended to $\Sigma_g$ real analytically,
and does not vanish at each point of $\Sigma_g$. 
\end{enumerate}
\end{Thm}

This suggests that surfaces with smooth mean 
curvature vector fields
form an important sub-class of 
the set of mixed type surfaces.

\section{Behavior of mean curvature 
along curves of type change
}
\label{sec:1}

Let $(M^3,g)$ be an oriented real analytic 
Lorentzian $3$-manifold.
Then, the vector product 
$\mb v\times_g \mb w$
is defined 
for linearly independent tangent vectors $\mb v,\mb w$
at $p\in M^3$, satisfying the 
following three properties:
\begin{enumerate}
\item $\mb v\times_g \mb w$
is orthogonal to $\mb v$ and
$\mb w$,
\item $\{\mb v, \mb w, \mb v\times_g \mb w\}$
is a basis of the tangent space $T_pM$
which is compatible with the
orientation of $M^3$,
\item it holds that
$$
g_p(\mb v\times_g \mb w,\mb v\times_g \mb w)
=-g_p(\mb v,\mb v)g_p(\mb w,\mb w)+g_p(\mb v,\mb w)^2.
$$
\end{enumerate}
For each tangent vector
$\mb v \in T_pM^3$ ($p\in M^3$),
we set
$$
|\mb v|:=\sqrt{|g_p(\mb v,\mb v)|}.
$$
We fix a domain $U$ in $\R^2$.
Let $f:U\to M^3$ be a real analytic immersion.
Set
$
f_{u}:=df(\partial_u)$,
$
f_{v}:=df(\partial_v)
$,
where 
$\partial_u:=\partial/\partial u$,
$\partial_v:=\partial/\partial v$.
Using three real analytic functions
$$
g_{11}:=g(f_u,f_u),\quad
g_{12}=g_{21}=  g(f_u,f_v),\quad
g_{22}:=g(f_v,f_v)
$$
on $U$, we define
a function $\beta:U\to \R$ by
\begin{equation}\label{def:beta}
\beta:=g_{11}g_{22}-g_{12}^2.
\end{equation}
Then
$$
U_+:=\{p\in U\,;\,\beta(p)>0\},\qquad
U_-:=\{p\in U\,;\,\beta(p)<0\}
$$
give the set of space-like points
and 
the set of time-like points, respectively.
The unit normal vector field 
\begin{equation}\label{eq:omega}
\nu=\frac{f_u\times_g f_v}{|f_u\times_g f_v|}
\end{equation}
of $f$
is well-defined on $U_+\cup U_-$.
Using this, we set
$$
h_{11}:=g(f_{uu},\nu),\quad
h_{12}=h_{21}=g(f_{uv},\nu),\quad
h_{22}:=g(f_{vv},\nu),
$$
where 
$$
f_{uu}=\nabla_{\partial_u}f_u,
\quad
f_{uv}=\nabla_{\partial_v}f_u=\nabla_{\partial_u}f_v,
\quad
f_{vv}=\nabla_{\partial_v}f_v,
$$
and $\nabla$ is the Levi-Civita connection
of the Lorentzian manifold $(M^3,g)$.
Each $h_{ij}$ ($i,j=1,2$)
is a function defined on $U_+\cup U_-$.
The mean curvature function $H$ 
is also defined on $U_+\cup U_-$,
and is given by
\begin{equation}\label{eq:H}
H:=
\frac{g_{11} h_{22}-2 g_{12} 
h_{12}+g_{22} h_{11}}{2|\beta|}
=\frac{\alpha}{2|\beta|^{3/2}},
\end{equation}
where
\begin{equation}\label{def:alpha}
\alpha:=\sqrt{|\beta|}(g_{11} h_{22}-2 g_{12} 
h_{12}+g_{22} h_{11}).
\end{equation}

Then the following assertion holds:

\begin{Lemma}\label{lem:A}
The function $\alpha:U_+\cup U_-\to \R$
can be analytically extended to $U$.
\end{Lemma}

\begin{proof}
We set
$\tilde \nu:=f_u\times_g f_v$.
Then 
$$
\beta=-g(f_u\times_g f_v,f_u\times_g f_v)
$$ 
and
$\nu=\tilde \nu/\sqrt{|\beta|}$ holds
(cf. \eqref{eq:omega}).
Therefore, we have that
\begin{align*}
\alpha&  =
\sqrt{|\beta|}
\biggl(g(f_{vv},\nu)g_{11}
-2g(f_{uv},\nu)g_{12}
+g(f_{uu},\nu)g_{22}\biggr) \\
&=
g(f_{vv},\tilde \nu)g_{11}
-2g(f_{uv},\tilde \nu)g_{12}
+g(f_{uu},\tilde \nu)g_{22},
\end{align*}
proving the assertion.
\end{proof}

Using the lemma, 
we now give the proof of
Theorem \ref{thm:main}:

\begin{proof}[Proof of Theorem \ref{thm:main}]
We may assume that the mean curvature function $H$ is not identically zero.
Let $(x^1,x^2)$ be the coordinates of $U$.
We fix a point
$p\in \overline{U_+} \cap \overline{ U_-}$.
Let $\epsilon>0$ be an arbitrary positive number
and $V$ a neighborhood of $p$.
It is sufficient to show that
there exist points $q_+\in V_+$
and $q_-\in V_-$ such that
$|H(q_+)|$ and $|H(q_-)|$ are both
less than $\epsilon$.
We may assume that $V$ is connected.
If $\beta\ge 0$ or $\beta \le 0$
on $V$, this contradicts
the fact that 
$p\in \overline{U_+} \cap \overline{ U_-}$.
So, we can take two points $q_0,q_1\in V$
such that $\beta(q_0)>0$ and $\beta(q_1)<0$.
We then take a smooth curve $\gamma(s)$
($s\in [0,2\pi]$) on $V$
such that $\gamma(0)=q_0$ and
$\gamma(2\pi)=q_1$.
Since the image of $\gamma$ lies in $V$,
we can write $\gamma=(\gamma^1,\gamma^2)$
and each $\gamma^i$ ($i=1,2$)
has the following Fourier series
expansion:
$$
\gamma^i(s)=u^i_0+\sum_{k=1}^\infty
\left( 
u_k^i \cos k s+v_k^i \sin k s
\right) \qquad (i=1,2).
$$
We then set
$$
\gamma^i_N(s)=u^i_0+\sum_{k=1}^N
\left( 
u_k^i \cos k s+v_k^i \sin k s
\right)\qquad (i=1,2),
$$
where $N$ is a
sufficiently
large positive integer.
Then the real analytic curve
defined by $\gamma_N(s):=(\gamma^1_N(s),\gamma^2_N(s))$
satisfies
\begin{equation}\label{eq:pm}
\beta(\gamma_N(0))>0,
\qquad
\beta(\gamma_N(2\pi))<0.
\end{equation}
Since 
$$
\hat\beta(s):=\beta(\gamma_N(s)) \qquad (0\le s \le 2\pi)
$$ 
is a real analytic function
defined on $[0,2\pi]$,
the set of zeros  of the function $\hat \beta(s)$
consists of a finite set of points
$$
0<s_1<\cdots<s_n<2\pi.
$$
By \eqref{eq:pm},
we can choose the number $j$
such that the sign of $\hat\beta(s)$ changes 
from positive to negative
at $s=s_j$.
Then there exists a positive integer $m$
such that
$$
\lim_{s\to s_j}\frac{\hat\beta(s)}{(s-s_j)^{m}}=b \,(\ne 0),
$$
where $b$ is a non-zero real number.
Since $\hat\beta(s)$ changes sign at $s=s_j$,
the integer $m$ is odd.
By Lemma \ref{lem:A}, we
may regard $\alpha$ as a real analytic
function on $U$.
So we set
$$
\hat\alpha(s):=\alpha(\gamma_N(s)).
$$
By \eqref{eq:H},
we have that
$$
H(\gamma_N(s)):=\frac{\hat \alpha(s)}{2|\hat \beta(s)|^{3/2}}
$$
for $s\ne s_1,...,s_n$.
Since $H$ is bounded,
we have
$
\hat\alpha(s_j)=0.
$
Since $\hat \alpha(s)$ is a real analytic function,
there exists a positive integer $\ell$
such that
$$
\lim_{s\to s_j}\frac{\hat \alpha(s)}{(s-s_j)^{\ell}}=a\, (\ne 0),
$$
where $a$ is a non-zero real number.
Then it holds that
$$
\lim_{s\to s_j}|s-s_j|^{(3m/2)-\ell}
|H(\gamma_N(s))|=
\frac{|a|}{|b|^{3/2}}\,(\ne 0).
$$
Since $H$ is bounded,  we have
$2\ell\ge 3m$.
Moreover, since $m$ is odd,
we have
$
2\ell> 3m.
$
Then we have
$
\lim_{s\to s_j}
|H(\gamma_N(s))|=0.
$
In particular, if we set
$$
q_+:=\gamma_N(s_j-\delta),\qquad
q_-:=\gamma_N(s_j+\delta),
$$
then $|H(q_+)|$ and $|H(q_-)|$
are less than $\epsilon$ for
sufficiently small $\delta>0$.
So we get the assertion.
\end{proof}

As a consequence,
we get the following corollary:

\begin{Cor}\label{cor:main}
Under the assumption of Theorem \ref{thm:main},
the function $\alpha:U_+\cup U_-\to \R$
can be analytically extended to $U$
and vanishes on $\overline{U_+}\cap \overline{U_-}$.
\end{Cor}

\begin{proof}
By Lemma \ref{lem:A},
the function $\alpha$ can be
analytically extended to $U$.
Suppose that
 $\alpha(p)\ne 0$ for 
$p\in \overline{U_+}\cap \overline{U_-}$.
Then the mean curvature function
cannot be bounded, since $\beta(p)=0$.
\end{proof}

We give here several examples:

\begin{Exa}[A space-like CMC
surface with parabolic symmetry]\
\label{ex:fp}
Consider the map
$f_P:\R^2\to \R^3_1$ such that
$$
f_P(u,v):=\left(-\eta(v) +u^2 v+v,
-\eta(v) +u^2 v-v,2 u v\right),
$$
where
$$
\eta(v):=
\frac{1}{2} \left(\arctan(v)-\frac{v}{v^2+1}\right),
\qquad \left|\arctan(v)\right|<\frac{\pi}2.
$$
This surface has singularities on
the $u$-axis. Moreover, the inverse image
$f^{-1}_P(\{\mb 0\})$ coincides with
the $u$-axis, where $\mb 0:=(0,0,0)$.
One can easily check that
$f_P$ gives a space-like immersion
of constant mean curvature $1/2$
on $\R^2\setminus \{v=0\}$.
Moreover, the image of $f_P$ 
is contained
in the set (cf. Figure 1, left)
$$
\mathcal P:=\left\{
(t,x,y)\in \R^3_1\,;\,
-t^2+x^2+y^2=2(t-x)\eta\left(\frac{t-x}2\right)
\right\}.
$$
The light-like line 
$$
L:=\{(c,c,0)\,;\,c\in \R \}
=\left\{\lim_{u\to \infty}f_P(u,\frac{c}{u^2})\,;\, 
c\in\R\right\}
$$
is contained in $\mathcal P$,
and the image of $f$ coincides with
$\mathcal P\setminus L$.
The set $\mathcal P$ itself is
a surface in $\R^3_1$
without self-intersections
which has a cone-like singular point
at the origin $\mb 0$, and
has bounded mean curvature function
on $\mathcal P\setminus\{\mb 0\}$.
Moreover, the induced metric on $\mathcal P$
degenerates only on the line $L$.
This implies that
we cannot drop the assumption that
$U_+,U_-$ are non-empty
in the statement of Theorem \ref{thm:main}.
This example is an analogue of
the maximal surface 
called {\it the Enneper surface
of the 2nd kind}
or
{\it parabolic catenoid}
(cf. \cite{K1}, \cite{FKKRSUYY1}).
\end{Exa}

\begin{figure}[htb]
\begin{center}
 \begin{tabular}{{c@{\hspace{20mm}}c}}
  \resizebox{3.8cm}{!}{\includegraphics{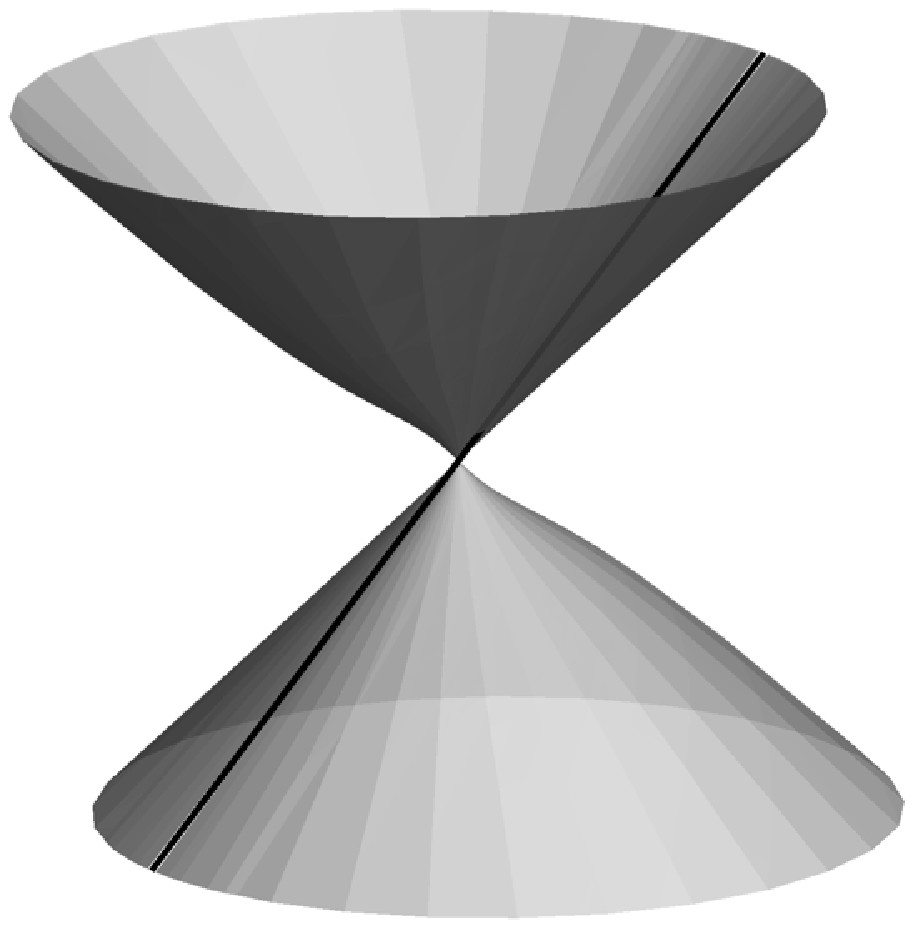}}&
  \resizebox{5.2cm}{!}{\includegraphics{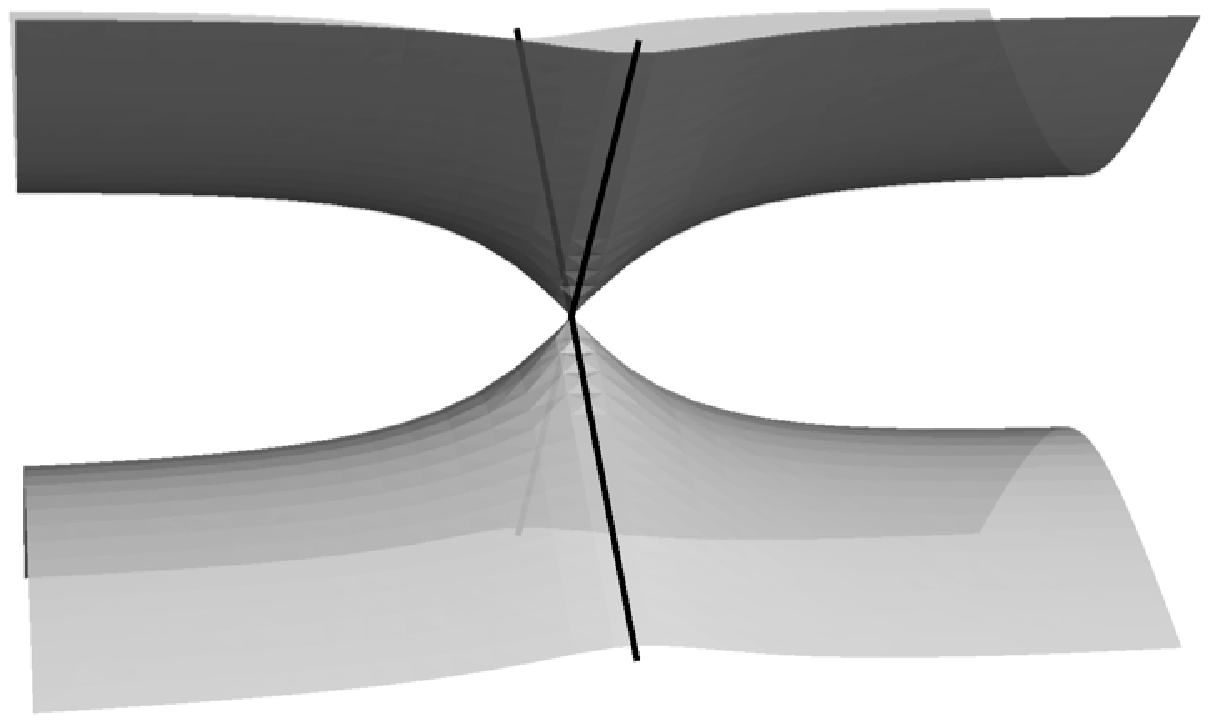}} 
 \end{tabular}
 \caption{
The figure of
 $\mathcal{P}$ (left) and $\mathcal{H}$ (right).
}
\label{fig:P_H}
\end{center}
\end{figure}

\begin{Exa}[A space-like CMC surface 
with hyperbolic symmetry]\
\label{ex:fh}
We next consider the map
defined by
$$
f_H(u,v):=(v \cosh u,v \sinh u,\phi(v))
\qquad ((u,v)\in \R\times (-1,1)),
$$
where
$$
\phi(v):=\log\left(\frac{1+v}{1-v}\right)-v.
$$
Like the case of $f_P$, this surface 
has singularities on the $u$-axis
and  $f^{-1}_H(\{\mb 0\})$ coincides with
the $u$-axis.
One can easily check that
$f_H$ gives a space-like immersion
of constant mean curvature $1/2$
on $\R^2\setminus \{v=0\}$.
Moreover, the image of $f_H$ 
is contained
in the set (cf. Figure 1, right)
$$
\mathcal H:
=\left\{
(t,x,y)\in \R^3_1\,;\,
y=\phi(\pm\sqrt{t^2-x^2})
\right\}
=
\left\{
(t,x,y)\in \R^3_1\,;\,
t^2=x^2+\psi(y)^2
\right\},
$$
where $\psi:\R\to (-1,1)$
is the inverse function of
$\phi:(-1,1)\to \R$.
Two light-like lines 
$$
L_\pm:=\{(c,\pm c,0)\,;\,c\in \R \}
$$
are contained in $\mathcal H$
and
$$
\mathcal H=L_+\cup L_- \cup 
(\mbox{Image of }f_H)\cup (\mbox{Image of }f'_H),
$$
where
$$
f'_H(u,v):=(-v \cosh u,v \sinh u,\phi(v))
\qquad ((u,v)\in \R\times (-1,1)).
$$
Like as in the case of $\mathcal P$,
the set $\mathcal H$ 
has no self-intersections,
and has bounded mean curvature function
on $\mathcal H\setminus\{\mb 0\}$.
 The origin $\mb 0$ is
a cone-like singular point.
Moreover, its induced metric
degenerates along the lines $L_\pm$.
This example is an analogue of
the maximal surface 
called  {\it the catenoid of the
2nd kind} 
or {\it hyperbolic catenoid} 
(cf. \cite{K1}, \cite{FKKRSUYY1}).
\end{Exa}

Similar examples, that is, a family of 
space-like surfaces
with constant mean curvature
one containing light-like lines
in the de Sitter 3-space $S^3_1$
have recently been found in \cite{FKKRUY}.

The following is one typical mixed type surface
whose mean curvature vanishes
identically.

\begin{Exa}\label{eq:Kob}
\label{ex:fk}
Consider the function
$$
f_K(x,y):=x \tanh y.
$$
Then the graph of $f_K$ in $\R^3_1$
gives a zero mean curvature surface,
which is space-like on the set
$U_+:=\{(x,y)\in \R^2\,;\, x^2>\cosh^2 y\}$
and time-like
on the set
$U_-:=\{(x,y)\in \R^2\,;\, x^2<\cosh^2 y\}$.
This example is called the {\it helicoid of the 2nd kind},
which was found by Kobayashi \cite{K1}.
\end{Exa}

On the other hand, we 
can find a similar example
in another space form:

\begin{Exa}
\label{ex:fds}
Consider the 
map $f_Z : \R \times S^1 \rightarrow S^3_1$
given by
\[
  f_Z(u,v):=\left( \sinh u \sin v,\, \cos u \cos v,\, \sin u \cos v,\, \cosh u \sin v \right),
\]
where 
$$
S^3_1:=\{(t,x,y,z)\in \R^4_1\,;\, -t^2+x^2+y^2+z^2=1\}
$$
is the de Sitter 3-space, which is
the space-time of
constant sectional curvature $1$.
Then the first fundamental form of
$f_Z$ is given by 
$ds^2=\cos 2v\, du^2 + dv^2$.
In particular, $f_Z$
%$\{(u,v)\in \R^2\,;\, \cos 2v>0\}$
%(resp. $\{(u,v)\in \R^2\,;\, \cos 2v>0\}$)
is space-like (resp. time-like)
if $\cos 2v>0$
(resp. $\cos 2v<0$).
Moreover, the mean curvature function of
$f_Z$ vanishes identically.
\end{Exa}

\begin{Exa}
\label{ex:fads}
We define an immersion $f_{\rm ads}:\R^2 \rightarrow H^3_1$ by
\[
  f_{\rm ads}(u,v)=
\left( \cosh u \cosh v,
\,\sinh a u\sinh v ,\,\cosh a u\sinh v ,
\,\sinh u \cosh v \right),
\]
where $a= 1/\tanh \alpha$ $(\alpha\neq0)$
is a constant, and
\begin{align*}
  H^3_1&=\left\{ (t,x,y,z)\in \R^4_2\,;\,
-t^2-x^2+y^2+z^2=-1 \right\}
\end{align*}
is the anti-de Sitter 3-space, which is
the space-time of
constant sectional curvature $-1$.
Then the first fundamental form
of $f_{\rm ads}$ is given by
$$
\dfrac{\cosh 2 \alpha - \cosh 2 v}{2\sinh^2\alpha} du^2 + dv^2.
$$
In particular,
$f_{\rm ads}$ is 
space-like (resp. time-like)
if $\cosh 2\alpha>\cosh 2v$
(resp. $\cosh 2\alpha<\cosh 2v$).
%a space-like 
%zero-mean curvature  immersion
%on $\{(u,v)\in \R^2\,;\,\cosh 2\alpha>\cosh 2v\}$
%and
%a time-like 
%zero-mean curvature  immersion
%on $\{(u,v)\in \R^2\,;\,\cosh 2\alpha<\cosh 2v\}$.
\end{Exa}

\section{Properties of points where
surfaces change type}
\label{sec:2}

In this section, we shall investigate
the properties of functions $t=f(x,y)$ 
whose graphs induce mixed type surfaces 
in 
 $\R^3_1$
with bounded mean curvature.

\begin{Def}[cf. {\cite[Definition 2.3]{FKKRSUYY2}}]
\label{def:nondegtc}
 Let $U$ be a domain
in the  $xy$-plane $\R^2$,
and $f:U\to \R$ a
$C^\infty$-function.
We set
$$
  B:=1-f_x^2-f_y^2.
$$
A point $p \in U$ 
is called a {\it non-degenerate point of  type change}
if 
\begin{equation}\label{eq:NB}
     B(p)=0,\qquad \nabla B(p)\ne0
\end{equation}
hold, where $\nabla B:=(B_x,B_y)$.
\end{Def}

By definition, the first fundamental
form of the graph of $f$ is degenerate
at a non-degenerate point of  type change.
We set 
$$
A := (f_x^2-1)f_{yy} -2 f_xf_y f_{xy}+(f_y^2-1)f_{xx}.
$$
Then the functions $A,B$ can be
considered as a 
special case of
the functions (cf.
\eqref{def:beta} and
 \eqref{def:alpha})
$\alpha, \beta$  
by setting $(u,v)=(x,y)$. 
By \eqref{eq:H0}, we have
\begin{equation}\label{eq:H2}
H=\frac{A}{2|B|^{3/2}}.
\end{equation}

\begin{Prop}[cf. Proposition 2.4 in \cite{FKKRSUYY2}]
\label{prop:Gu}
\label{prop:equiv}
Suppose that the mean curvature function of the
graph of $f$ is bounded.
Let
$p\in U$ be a point satisfying
$B(p)=0$. Then the following two
assertions are equivalent:
\begin{enumerate}
\item\label{item:equiv:1} 
the point $p$ is a
non-degenerate point of  type change.
\item\label{item:equiv:2}
$p$ is a
dually regular point in the sense of
\cite{G}, that is,
$p$ is a point where 
$f_{xx}(p)f_{yy}(p)-f_{xy}(p)^2\ne 0$ .
\end{enumerate}
\end{Prop}

\begin{proof}
The proof is almost parallel
to that of
Proposition 2.4 in \cite{FKKRSUYY2}.
It holds that
\begin{equation}\label{eq:Bxy}
\nabla B
=\op{Hess}(f)\pmt{f_x \\ f_y},
\qquad \op{Hess}(f)
:=\pmt{f_{xx} & f_{xy}\\ f_{xy} & f_{yy}}.
\end{equation}
Now suppose that (2) holds. 
Then $\op{Hess}(f)$ is a regular matrix at $p$.
Since $B(p)=0$, $(f_x,f_y)\ne 0$ at $p$.
Thus, \eqref{eq:Bxy} implies that $\nabla B\ne 0$
at $p$, that is, (1) holds.

We next suppose on the contrary that (2) 
does not hold. By a suitable linear
coordinate change of $(x, y)$, 
we may assume without loss of generality that
$f_{xy}(p) = 0$. 
Then either $f_{xx}(p) = 0$ or $f_{yy}(p) = 0$. 
By \eqref{eq:H2}, and
Theorem \ref{thm:main}, we have
$A(p)=0$.
This with $B(p) = 0$ and $f_{xy}(p) = 0$ 
implies that
$$
f_x(p)^2f_{xx}(p) + f_y(p)^2f_{yy}(p) = 0.
$$
This with $f_{xx}(p)=0$ or $f_{yy}(p)=0$ implies that
$$
\op{Hess}(f)\pmt{f_x \\ f_y}
=\pmt{
f_x(p)f_{xx}(p)\\
f_y(p)f_{yy}(p)
}
=\pmt{0 \\ 0}.
$$
So (1) does not hold.
\end{proof}

A regular curve $\Gamma:(a,b)\to \R^3_1$ 
is called 
{\it null\/} or {\it isotropic\/}
if
$\dot \Gamma(t):=d\Gamma(t)/dt$
is a light-like vector
for each $t\in (a,b)$.

\begin{Def}\label{def:n-deg}
A null curve $\Gamma:(a,b)\to \R^3_1$ 
is called non-degenerate at
$t=c$ if
$\dot \Gamma(c)$ and $\ddot \Gamma(c)$ 
are linearly independent.
If $\Gamma(t)$ is non-degenerate
for all $t\in (a,b)$,
the curve
$\Gamma$
is called a {\it non-degenerate} null curve.
\end{Def}

Let $p\in U$
be a non-degenerate point of type change.
Then, by the implicit function theorem,
there exists a regular curve
$
\gamma:(-\epsilon,\epsilon)\to U
$
such that
$B\circ \gamma(t)=0$
and
$\gamma(0)=p$, where $\epsilon$ is a
positive number.
We call this curve $\gamma$
{\it the characteristic curve}
of type change.
The following assertion
is a generalization of
\cite[Proposition 2.5]{FKKRSUYY2}
for zero-mean curvature surfaces. 

\begin{Prop}\label{prop:1}
Suppose that the graph $t=f(x,y)$ 
over a domain $U$ 
has bounded mean curvature function.
If the graph changes type
along a regular curve $\gamma(t)$
$(|t|<\epsilon)$
such that $f\circ \gamma(t)$
is a non-degenerate null curve in $\R^3_1$,
then $\gamma(t)$
consists of non-degenerate points
of type change.
\end{Prop}

\begin{proof}
The proof is completely parallel
to that of \cite[Proposition 2.5]{FKKRSUYY2}.
\end{proof}

The converse assertion is given
as follows,
which is a generalization of
\cite[Proposition 2.6]{FKKRSUYY2}
for zero-mean curvature surfaces. 

\begin{Prop}
Suppose that the graph $t=f(x,y)$ 
over a domain $U$ 
has bounded mean curvature function.
Let $p\in U$ be a
non-degenerate point of type change
and $\gamma(t)$
$(|t|<\epsilon)$
the characteristic curve
of type change such that $\gamma(0)=p$.
Then $f\circ \gamma(t)$
is a non-degenerate null curve.
\end{Prop}

\begin{proof}
Using the fact that
$A(\gamma(t))=0$ holds,
the proof of this assertion
is completely parallel to
that of \cite[Proposition 2.6]{FKKRSUYY2}.
\end{proof}

Moreover, the following assertion holds:

\begin{Prop}\label{prop:tc}
Let $t=f(x,y)$ 
be a real analytic
function over the domain $U$ 
which gives a graph with
bounded mean curvature function.
Suppose that the zeros of $B(x,y)$ are all
non-degenerate points of type change.
Then, the mean curvature
vector $H\nu$ 
can be analytically
extended 
to all of 
$U$.
\end{Prop}

\begin{proof}
Let $p\in U$ be 
a non-degenerate point of type
change. Then we can take
a real analytic local coordinate system
$(u,v)$ centered at $p$
such that the $u$-axis is the
characteristic curve
of type change.
By the condition $\nabla B(u,0)\ne (0,0)$
(cf. \eqref{eq:NB}),
there exists a real analytic function
$b(u,v)$ defined near the $u$-axis
such that $B(u,v)=v b(u,v)$ and
$b(u,0)\ne  0$.
On the other hand, Theorem \ref{thm:main}
yields that 
there exists a real analytic function
$a(u,v)$ defined near the $u$-axis
such that 
\begin{equation}\label{eq:Aa}
A(u,v)=v^2 a(u,v).
\end{equation}
By \eqref{eq:H2}, we have
$$
H(u,v)=\frac{\sqrt{|v|} a(u,v)}{2 |b(u,v)|^{3/2}}.
$$
By \eqref{eq:nu}, we have that
\begin{equation}\label{eq:Hnu}
H\nu
=
\frac{\sqrt{|v|}a(u,v)}{2 |b(u,v)|^{3/2}}
\frac{1}{\sqrt{|v| |b(u,v)|}}(1,f_x,f_y)
=
\frac{a(u,v)}{2b(u,v)^{2}}
(1,f_x,f_y),
\end{equation}
proving the assertion.
\end{proof}

Finally, we prove
Theorem \ref{thm:existence}
in the introduction:

\begin{proof}[Proof of Theorem \ref{thm:existence}]
Let  $f:\R^2\to \R$  be a real analytic function whose graph gives
a zero-mean curvature surface, with
function $B:=1-f_x^2-f_y^2$
satisfying $\nabla B\ne (0,0)$ if $B=0$. 
Take a real analytic function $\psi:\R\to \R$
such that
\begin{equation}\label{eq:psi}
\psi(0)=\psi'(0)=\psi''(0)=0.
\end{equation}
We then set
\begin{equation*}
 g(x,y):= f(x,y) + \psi (B(x,y)),
\end{equation*}
and
$$
\tilde B:=1-g_x^2-g_y^2.
$$
Since
$$
 g_x = f_x + \psi'(B)B_x, \quad
 g_y = f_y + \psi'(B)B_y,
$$
we have that
\begin{equation}\label{eq:BtB}
\tilde B=B-2\psi'(B)(f_xB_x+f_yB_y)-\psi'(B)^2(B_x^2+B_y^2).
\end{equation}
Here, the relation $C_1\equiv C_2 \mod B$ for
two real analytic functions
$C_i(x,y)$ ($i=1,2$)
means that $(C_1-C_2)/B$ 
is a real analytic function on $\R^2$.
Since $\psi'(B)\equiv 0 \mod B$,
$\tilde B$ can be divided by $B$.
Thus, to show the mean curvature vector field
can be smoothly extended across the set $B=0$, 
it is sufficient to show
that 
$$
\tilde A:=
(g_y^2-1)g_{xx}
-2g_xg_yg_{xy}
+(g_x^2-1)g_{yy}
$$
can be divided by $B^2$.
Since
\begin{align*}
g_{xx}
&= f_{xx} + \psi''(B)B_x^2 + \psi'(B)B_{xx},
\\
g_{xy}&
= f_{xy} + \psi''(B)B_xB_y
+ \psi'(B)B_{xy} , \\
g_{yy} &
= f_{yy} + \psi''(B)B_y^2
+ \psi'(B)B_{yy},
\end{align*}
the fact that $A=0$ yields 
\begin{equation}\label{eq:tA3}
\tilde A\equiv  \psi''(B)\Gamma+\psi'(B)\Delta \mod B^3,
\end{equation}\label{eq:tA}
where
\begin{align*}
\Gamma&:=( f_y^2-1)B_{x}^2
-2f_x f_yB_xB_y
+(f_x^2-1)B_y^2, \\
\Delta&:=
2 (B_x f_x f_{yy}
-B_x f_{xy} f_y-B_y f_x f_{xy}+
B_y f_{xx} f_y) \\
&\phantom{aaaaaaaaaaaaaaaaaa}
+B_{xx} \left(f_y^2-1\right)
-2B_{xy} f_x f_y+
B_{yy} 
\left(f_x^2-1\right).
\end{align*}
Since
\begin{align*}
\Gamma&=
(-B-f_x^2)B_{x}^2
-2f_xf_yB_xB_y
+(-B-f_y^2)B_y^2 \\
&=-B(B_x^2+B_y^2)
 -(f_xB_x+f_yB_y)^2 
\end{align*}
and
\begin{align}
f_xB_x+f_yB_y &=-2 \nonumber
(f_x(f_xf_{xx}+f_yf_{xy})
+f_y(f_xf_{xy}+f_y f_{yy}))\\
&=2A +2B (f_{xx}+f_{yy})
=2B (f_{xx}+f_{yy}) \label{eq:tB},
\end{align}
we have that
\begin{equation}\label{eq:Gamma}
\Gamma\equiv -B(B_x^2+B_y^2) \mod B^2.
\end{equation}
Since 
$$\psi'(B)\equiv 0\mod B^2,\qquad
\psi''(B)\equiv 0\mod B,
$$
\eqref{eq:tA3} and \eqref{eq:Gamma}
yield that
$
\tilde A
$
can be divided by $\tilde B^2$.

To give an explicit example,
we consider the function $f_K(x,y):=x \tanh y$
given in Example \ref{eq:Kob}. 
Then, we have
$$
B(x,y)=(\cosh^2 y - x^2)\text{sech}^4 y
$$
and $x=\pm \cosh y$ 
give the characteristic curves of type change.
We consider the new function
\begin{equation}\label{eq:g0}
g(x,y):=x \tanh y+ c\tanh^3(B(x,y))\qquad (0< c\le 1),
\end{equation}
where $c$ is a constant.
Then the mean curvature vector field
is real analytic along the set
of type change
$\Sigma_f:=\{(\pm \cosh y,y)\,;\, y\in \R\}$.

By \eqref{eq:Gamma},
$$
\Gamma\equiv
-4B \text{sech}^4y 
\mod B^2
$$
holds. By a straightforward calculation, 
$$
\Delta=2(B+1)\op{sech}^4 y
$$
holds. Since $\psi(B)=c \tanh^3(B)$, we have
$$
\psi'(B)\equiv 3cB^2,\quad \psi''(B)\equiv 6cB
\mod B^3.
$$
Thus, \eqref{eq:tA3} yields that
\begin{equation}\label{eq:ratio}
\left.
\frac{\tilde A}{\tilde B^2}
\right|_{(x,y)=(\pm \cosh y,y)}
=
\left.
\frac{\tilde A}{B^2}
\right|_{(x,y)=(\pm \cosh y,y)}
=
\frac{-18c}{\cosh^4 y}
\qquad (y\in \R),
\end{equation}
which never vanishes on the set $\Sigma_f$.

To complete the proof, it is
sufficient to show that $\tilde B/B$
has no zeros if $c$ is sufficiently small. 
We shall now compute $\tilde B/B$
using \eqref{eq:BtB}.
We set
$$
\phi(t):=\frac{\tanh t}t 
$$
which is a real analytic bounded function.
We set
$$
U:=x \,\op{sech}^2y,\quad V:=\op{sech}\, y,\quad
S:=\op{sech}(V^2-U^2).
$$
Here $U$ is unbounded, but $V, S$ are bounded on
$\R^2$. By a straight-forward calculation,
one can get that
$$
\frac{\tilde B}{B}
=1-12 c B \phi(B)^2 S^2 (C_1+C_2),
$$
where
\begin{align*}
C_1&:=2 U(U^2-V^2)\tanh y, \\
C_2&:=3c B^2 \phi(B)^2S^2
\biggl(U^2V^4+(-2U^2+V^2)^2\tanh^2 y\biggr).
\end{align*}
Since 
$$
\frac{\cosh(V^2-U^2)}{\cosh(U^2)}
=\cosh(V^2)-\sinh(V^2)\tanh(U^2),
$$
using the fact that $|V|\le 1$, we have
$$
e^{-1}\le \exp(-V^2)=\cosh(V^2)-\sinh(V^2)<
\frac{\cosh(V^2-U^2)}{\cosh(U^2)}.
$$
In particular
$$
S|U|^m 
=\frac{|U|^m}{\cosh(U^2)}\frac{\cosh(U^2)}{\cosh(V^2-U^2)}
<\frac{e|U|^m}{\cosh(U^2)}
$$
is a bounded function for $m\ge 0$.
Then we can write
$$
\frac{\tilde B}{B}
=1-12 c \phi(B)^2 SB  (SC_1+SC_2).
$$
Since 
$\tanh y$, $\phi(B)$, and $SB=B\op{sech} B$  
are all bounded,
there exists a positive constant $m$ which
does not depend on the choice of $c\in (0,1]$
such that
 $\phi(B)^2SB(SC_1+SC_2)<m$ holds 
for all $(x,y)\in \R^2$, and so
$$
\left|\frac{\tilde B}{B}-1\right|<12 mc.
$$
If $0<c<1/(12 m)$, then the zero set of
$\tilde B$ coincides with that of $B$, proving the 
assertion.
\end{proof}

\begin{ack}
The first, the fourth and the fifth authors 
thank Udo Hertrich-Jeromin and Kosuke Naokawa
for fruitful conversations at TU-Wien.
The authors thank Wayne Rossman for
valuable comments.
\end{ack}

\end{document}